\documentclass{article}
\usepackage[T1]{fontenc}
\usepackage[utf8]{inputenc}
\usepackage{amsmath}
\usepackage{amsthm}
\usepackage{tikz}
\usepackage{fullpage} 
\usepackage{multicol}
\usepackage{stmaryrd}
\usepackage[colorlinks=true,linkcolor=black,anchorcolor=black,citecolor=black,filecolor=black,menucolor=black,runcolor=black,urlcolor=black]{hyperref}

\newtheorem{theorem}{Theorem}     

\newtheorem{lemma}[theorem]{Lemma}

\newtheorem{wpgc}[theorem]{Wegner's Planar Graph Conjecture}

\newtheorem{claim}{Claim}

\newcommand{\vp}{\varphi}

\title{Relaxation of Wegner's Planar Graph Conjecture for maximum degree 4}

\author{Eun-Kyung Cho\thanks{
Department of Mathematics, Hankuk University of Foreign Studies, Yongin-si, Gyeonggi-do, Republic of Korea.
 \texttt{ekcho2020@gmail.com}
}
\and Ilkyoo Choi\thanks{
Department of Mathematics, Hankuk University of Foreign Studies, Yongin-si, Gyeonggi-do, Republic of Korea.
\texttt{ilkyoo@hufs.ac.kr}
}
\and Bernard Lidick\'y\thanks{
Department of Mathematics, Iowa State University, Ames, IA, USA.
\texttt{lidicky@iastate.edu}
}}
\date\today

\begin{document}

\maketitle

\begin{abstract}
The famous Wegner's Planar Graph Conjecture asserts tight upper bounds on the chromatic number of the square $G^2$ of a planar graph $G$, depending on the maximum degree $\Delta(G)$ of $G$. 
The only case that the conjecture is resolved is when $\Delta(G)=3$, which was proven to be true by Thomassen, and independently by Hartke, Jahanbekam, and Thomas.
For $\Delta(G)=4$, Wegner's Planar Graph Conjecture states that the chromatic number of $G^2$ is at most 9; even this case is still widely open, and very recently Bousquet, de Meyer, Deschamps, and Pierron claimed an upper bound of 12.

We take a completely different approach, and show that a relaxation of properly coloring the square of a planar graph $G$ with $\Delta(G)=4$ can be achieved with 9 colors. 
Instead of requiring every color in the neighborhood of a vertex to be unique, which is equivalent to a proper coloring of $G^2$, we seek a proper coloring of $G$
such that at most one color is allowed to be repeated in the neighborhood of a vertex of degree 4, but nowhere else.  
\end{abstract}

\section{Introduction}

Given a graph $G$, let $V(G)$ and $E(G)$ denote the set of vertices and the set of edges, respectively, of $G$.
For each $v \in V(G)$, the {\it neighborhood} of $v$, denoted $N_G(v)$, is the set of  vertices adjacent to $v$, and the {\it degree} of $v$, denoted $d_G(v)$, is the number of neighbors of $v$.
A {\it proper coloring} $\vp$ of a graph $G$ assigns colors to vertices of $G$ so that $\vp(x)\neq\vp(y)$ for every edge $xy$ of $G$. 

Given a graph $G$, the {\it square} of $G$, denoted $G^2$, is the graph obtained from $G$ by adding edges between every pair of vertices at distance $2$. 
The famous and very popular Wegner's Planar Graph Conjecture~\cite{Wegner}, first raised in 1977, asserts tight upper bounds on the chromatic number of the square $G^2$ of a planar graph $G$, depending on the maximum degree $\Delta(G)$ of $G$. 
We state the conjecture below, and refer the readers to~\cite{2017CrWe} for illustrations of the tightness examples. 

\begin{wpgc}[\cite{Wegner}]
If $G$ is a planar graph, then 
\[\chi(G^2) \le
\begin{cases}
7 & \textrm{if } \Delta(G) = 3,\\
\Delta(G)+5 & \textrm{if } \Delta(G)\in\{4,5,6,7\}, \\
 \frac{3}{2}\Delta(G)  +1 & \textrm{if } \Delta(G) \ge 8.
\end{cases}\]
\end{wpgc}

For sufficiently large maximum degree, Havet, van den Heuvel, McDiarmid, and Reed~\cite{2007HaHeMcRe} proved that the above conjecture is true asymptotically. 
For exact results, Molloy and Salavatipour~\cite{2005MoSa} proved the current best bound. 

\begin{theorem}[\cite{2005MoSa}]
If $G$ is a planar graph, then $\chi(G^2)\leq \left\lceil\frac{5}{3}\Delta(G)\right\rceil+78$. 
\end{theorem}

The only case that the conjecture is resolved is when $\Delta(G)= 3$, which was proven to be true by Thomassen~\cite{2018Thomassen}, and independently by Hartke, Jahanbekam, Thomas~\cite{WegnerHartke}; 
the former proof uses a meticulous induction argument, and the latter uses a simple discharging argument with a computer assisted proof of its reducible configurations. 

For planar graphs with maximum degree at least $4$, Wegner's Planar Graph Conjecture is still wide open, and we refer the readers to the following references for various partial results, oftentimes with lower bound constraints on the maximum degree~\cite{2003AgHa,unfound_BoBrGlHe,arXiv_BoDeMePi_improved,thesis_Jonas,arXiv_KrRzTu,2002MaMa,2003HeMc,unfound_Wong,2018ZhBu}. 
In particular, when the maximum degree is exactly $4$, after a series of improvements in~\cite{2014CrErSk,2016CrRa,2022ZhBuZh} by various authors, Bousquet, de Meyer, Deschamps, and  Pierron~\cite{arxiv_BoMeDePi_square} very recently claimed to have established an upper bound of 12. 
Note that the conjectured upper bound is 9. 


\medskip

In this paper, we take a completely different approach, and show that a relaxation of coloring the square of a planar graph with maximum degree 4 can be achieved with 9 colors. 
Instead of requiring every color in the neighborhood of a vertex to be unique, which is equivalent to a proper coloring of $G^2$, we seek a proper coloring of $G$ such that at most one color is allowed to be repeated in the neighborhood of a vertex of degree 4, but nowhere else. 
In other words, every vertex $v$ has at least $\min\{2, d(v)\}$ colors appearing exactly once in its neighborhood.
Note that requiring $\min\{3, d(v)\}$ unique colors in the neighborhood of every vertex $v$ is equivalent to a proper coloring of the square of the graph when it has maximum degree 4. 
We now state our main result: 

\begin{theorem}\label{thm:main}
Every planar graph has a proper 9-coloring such that each neighborhood of a vertex $v$ has at least $\min\{2, d(v)\}$ unique colors.
In other words, every planar graph has a proper 9-coloring such that at most one color is allowed to be repeated in the neighborhood of a vertex of degree $4$, but nowhere else.  
\end{theorem}




An {\it $h$-PCF $k$-coloring} $\vp$ of a graph $G$ is a proper $k$-coloring of $G$ such that each neighborhood of every vertex $v$ has at least $\min\{h, d(v)\}$ unique colors. 
This concept is a generalization of proper conflict-free coloring, defined recently by Fabrici, Lu\v zar, Rindo\v sov\'a, and Sot\'ak~\cite{fabrici2022proper}, see also~\cite{arXiv_CaPeSk, arXiv_Hickingbotham, arXiv_Liu, conflictfree}.

A {\it $k$-vertex}, {\it $k^-$-vertex}, {\it $k^+$-vertex} is a vertex of degree $k$, at least $k$, at most $k$, respectively. 

Given a vertex $v$ of a graph $G$ with a $2$-PCF coloring $\vp$, the {\it unique colors of $v$} are the unique colors appearing in the neighborhood of $v$; in particular, let $\vp_1(v)$ and $\vp_2(v)$ denote two unique colors of $v$, if they exist. For $X \subseteq V(G)$, we abuse the notation and define $\vp(X) = \{\vp(v):v\in X\}$.


For $S \subseteq V(G)$ where each vertex in $S$ has at most two neighbors not in $S$, define $G*S$ to be the graph obtained from $G$ by removing $S$ and adding an edge $uv$ for $u, v\in V(G)\setminus S$ if $u$ and $v$ have a common neighbor in $S$ and $uv$ is not an edge already;
$G*S$ is called the {\it $S$-reduced graph}.
Note that $G*S$ is planar whenever $G$ is planar, and the maximum degree of $G*S$ is at most the maximum degree of $G$.

For a $2$-PCF coloring $\vp$ of $G*S$, let $v \in S$ and $u \in N_G(v) \setminus S$.
If vertices in $N_G(u) \setminus S$  receive distinct colors (in particular if $u$ is a $3$-vertex), then let $B_{S}(u) = \{\vp(u), \vp_1(u), \vp_2(u)\}$. 
(If either $\vp_1(u)$ or $\vp_2(u)$ is not defined, then ignore it.)
If there is a repeated color among vertices in $N_G(u) \setminus S$, then let $B_S(u) = \{\vp(u)\} \cup \vp(N_{G-S}(u))$.
Notice that for $u \in V(G*S)$ with a neighbor in $S$
\begin{align}\label{eq:3}
    |B_S(u)| \le 3
\end{align}
if $G$ has maximum degree at most $4$.
Let $C_{G*S}(v) = \bigcup_{u \in N_{G}(v) \setminus S} B_S(u)$.
By \eqref{eq:3}, $|C_{G*S}(v)| \le 3|N_{G}(v) \setminus S|$ when $G$ has maximum degree at most $4$. 
Moreover, if $\vp$ assigns a color not in $C_{G*S}(v)$ to $v$, then two unique colors are guaranteed for vertices in $N_G(v) \setminus S$ and $\vp$ is still a (partial) proper coloring.

\section{Proof of Theorem~\ref{thm:main}}


Let $G$ be a counterexample to Theorem~\ref{thm:main} with the minimum number of vertices. 
We first prove a sequence of claims regarding the structure of $G$.

\begin{claim}\label{clm:3vx}
$G$ does not have a $2^-$-vertex. 
\end{claim}
\begin{proof}
Let $v$ be a vertex of minimum degree in $G$. 
Suppose $v$ is a $2^-$-vertex. 
For $S=\{v\}$, let $H$ be the $S$-reduced graph. 
By the minimality of $G$, $H$ has a $2$-PCF $9$-coloring $\vp$.
Extend $\vp$ to all of $G$ by 
coloring $v$ with a color not in $C_H(v)$.
Now $\vp$ is a $2$-PCF $9$-coloring of $G$, which is a contradiction.
\end{proof}

\begin{figure}[h]
\centering
\begin{tikzpicture}
\draw (90:2) -- (90:1);
\draw (90:1) -- (-30:1) -- (210:1) -- cycle;
\draw (-30:1) -- (-20:2);
\draw[dotted] (-30:1) -- (-40:2);
\draw (210:1) -- (200:2);
\draw[dotted] (210:1) -- (220:2);
\draw (90:2) to [shift={(90:2)}](135:0.7);
\draw[dotted] (90:2) to [shift={(90:2)}](45:0.7);
\draw (90:2) to [shift={(90:2)}](90:0.6);
\filldraw[fill=black, draw=black] (90:1) circle (0.05) node [left] {$x$};
\filldraw[fill=white, draw=black] (90:2) circle (0.05) node [left] {$x_1$};
\filldraw[fill=white, draw=black] (210:1) circle (0.05) node [above left] {$y$};
\filldraw[fill=white, draw=black] (200:2) circle (0.05) node [left] {$y_1$};
\filldraw[fill=white, draw=black] (220:2) circle (0.05) node [left] {$y_2$};
\filldraw[fill=white, draw=black] (-30:1) circle (0.05) node [above right] {$z$};
\filldraw[fill=white, draw=black] (-20:2) circle (0.05) node [right] {$z_1$};
\filldraw[fill=white, draw=black] (-40:2) circle (0.05) node [right] {$z_2$};
\draw (0,-1.5) node{(a)};
\end{tikzpicture}
\qquad \qquad \qquad
\begin{tikzpicture}
\draw (90:1) -- (100:2);
\draw (90:1) -- (80:2);
\draw (90:1) -- (-30:1) -- (210:1) -- cycle;
\draw (-30:1) -- (-20:2);
\draw (-30:1) -- (-40:2);
\draw (210:1) -- (200:2);
\draw (210:1) -- (220:2);
\filldraw[fill=black, draw=black] (90:1) circle (0.05) node [left] {$x$};
\filldraw[fill=white, draw=black] (100:2) circle (0.05) node [left] {$x_1$};
\filldraw[fill=white, draw=black] (80:2) circle (0.05) node [right] {$x_2$};
\filldraw[fill=black, draw=black] (210:1) circle (0.05) node [above left] {$y$};
\filldraw[fill=white, draw=black] (200:2) circle (0.05) node [left] {$y_1$};
\filldraw[fill=white, draw=black] (220:2) circle (0.05) node [left] {$y_2$};
\filldraw[fill=black, draw=black] (-30:1) circle (0.05) node [above right] {$z$};
\filldraw[fill=white, draw=black] (-20:2) circle (0.05) node [right] {$z_1$};
\filldraw[fill=white, draw=black] (-40:2) circle (0.05) node [right] {$z_2$};
\draw (0,-1.5) node{(b)};
\end{tikzpicture}
\caption{A $3$-cycle with a $3$-vertex and a $3$-cycle with no $3$-vertex}
\label{fig:triangle3}
\end{figure}
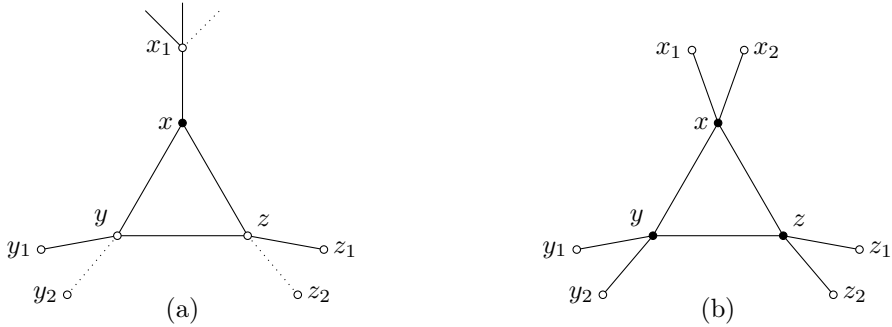

\begin{claim}\label{clm:triangle3}
$G$ does not have a $3$-cycle with a $3$-vertex.
\end{claim}

\begin{proof}
Suppose $G$ contains a $3$-cycle $T:xyz$ where $x$ is a $3$-vertex. 
Let $x_1, y_1, z_1$ be neighbors of $x,y,z$, respectively, not on $T$. 
Let $y_2$ (resp. $z_2$) be the neighbor of $y$ (resp. $z$) that is neither on $T$ nor $y_1$ (resp. $z_1$) if $y$ (resp. $z$) is a $4$-vertex. 
See Figure~\ref{fig:triangle3}(a).

Suppose $x_1$ is a $3$-vertex. 
For $S = \{x, x_1\}$, let $H$ be the $S$-reduced graph.
By the minimality of $G$, $H$ has a $2$-PCF $9$-coloring $\vp$.
Now extend $\vp$ to all of $G$ as follows: color $x_1$ with a color not in $C_{H}(x_1) \cup \{\vp(y), \vp(z)\}$ to guarantee two (actually three) unique colors for $x$, and color $x$ with a color not in $\{\vp(x_1), \vp_1(x_1), \vp_2(x_1), \vp(y), \vp(z), \vp(y_1), \vp(z_1)\}$ to guarantee two  (actually three) unique colors for $x_1$. 
Thus $\vp$ is a $2$-PCF $9$-coloring of $G$, which is a contradiction. 

Now we know $x_1$ is a $4$-vertex. 
For $S'=\{x,y,z\}$, let $H'$ be the $S'$-reduced graph.
By the minimality of $G$, $H'$ has a $2$-PCF $9$-coloring $\vp'$.
Since $x_1$ has three unique colors at this point, at least two unique colors for $x_1$ are guaranteed regardless of the color assigned to $x$. 

Suppose $y$ is a $3$-vertex. 
Then color $z$ with a color not in $C_{H'}(z) \cup \{\vp'(x_1), \vp'(y_1)\}$, and color $y$ with a color not in $C_{H'}(y) \cup \vp'(N_{G}(z) \setminus S') \cup \{\vp'(x_1), \vp'(z)\}$.
At this point $x$ has two (actually three) unique colors. 
Color $x$ with a color not in $\{\vp'(x_1), \vp'(y), \vp'(z), \vp'(y_1)\} \cup \vp'(N_{G}(z) \setminus S')$ to guarantee two unique colors for each of $y$ and $z$
Thus $\vp'$ is a $2$-PCF $9$-coloring of $G$, which is a contradiction. 

By symmetry, we may assume both $y$ and $z$ are $4$-vertices. Now, color $y$ with a color not in $C_{H'}(y) \cup \{\vp'(x_1)\}$, and color $z$ with a color not in $C_{H'}(z) \cup \{\vp'(y), \vp'(x_1)\}$.
This guarantees two (actually three) unique colors for $x$. 
Color $x$ with a color not in $\{\vp'(x_1), \vp'(y), \vp'(y_1), \vp'(y_2), \vp'(z), \vp'(z_1), \vp'(z_2)\}$ to guarantee an additional unique color for each of $y$ and $z$.
Note that each of $y$ and $z$ already had a unique color in $N_G(y)\setminus\{x\}$ and $N_G(z)\setminus\{x\}$, respectively, since $H'$ is an $S'$-reduced graph. 
Then $\vp'$ is a $2$-PCF $9$-coloring of $G$, which is a contradiction.

\end{proof}

\begin{claim}\label{clm:triangle}
$G$ does not have a $3$-cycle.
\end{claim}

\begin{proof}
Suppose $G$ contains a $3$-cycle $T:xyz$.
By Claim~\ref{clm:triangle3}, all vertices on $T$ are $4$-vertices. 
Let $x_1, x_2$, and $y_1, y_2$, and $z_1, z_2$ be the neighbors of $x$ and $y$ and $z$, respectively, not on $T$. 
See Figure~\ref{fig:triangle3}(b).
For $S = \{x,y,z\}$, let $H$ be the $S$-reduced graph.
By the minimality of $G$, $H$ has a $2$-PCF $9$-coloring $\vp$.
Let $C'=C_{H}(x) \cup \{\vp(y_1),\vp(y_2),\vp(z_1),\vp(z_2)\}$.

Suppose $|C'| \le 8$.
First color $x$ with a color not in $C'$ to guarantee three unique colors for each of $y$ and $z$, so at least two unique colors are guaranteed for $y$ and $z$ regardless of the colors assigned to $y$ and $z$. 

If $|C_{H}(y) \cup \{\vp(x), \vp(x_1), \vp(x_2)\}| \le 8$, then color $y$ with a color not in $C_{H}(y) \cup \{\vp(x), \vp(x_1), \vp(x_2)\}$, guaranteeing three unique colors for $x$, so at least two unique colors are guaranteed for $x$ regardless of the color assigned to $z$. 
Now color $z$ with a color not in $C_{H}(z) \cup \{\vp(x), \vp(y)\}$.
Now, $\vp$ is a $2$-PCF $9$-coloring of $G$, which is a contradiction.

Thus, by symmetry, we may assume $|C_{H}(y) \cup \{\vp(x), \vp(x_1), \vp(x_2)\}| = |C_{H}(z) \cup \{\vp(x), \vp(x_1), \vp(x_2)\}| = 9$.
Without loss of generality, assume $\vp(x_i) = i$ for $i \in \{1,2\}$, $\vp(x)=3$, and $C_{H}(y)=C_{H}(z)=\{4,5,6,7,8,9\}$.
Delete the color on $x$ and color $y$ with $3$ and $z$ with $1$ to guarantee two unique colors for $x, y, z$.
Now color $x$ with a color not in $C_{H}(x) \cup \{\vp(y), \vp(z)\}$ to obtain a $2$-PCF $9$-coloring of $G$, which is a contradiction.

Now we know, $|C'|=9$, so either $\vp(y_1)$ or $\vp(y_2)$ appears only once on $N_G(\{x_1,x_2,x,y,z\})\setminus S$.
Without loss of generality, assume $\vp(y_1)$ appears only once on $N_G(\{x_1,x_2,x,y,z\})\setminus S$.
Color $x$ with $\vp(y_1)$, guaranteeing the three unique colors for $z$.
Color $z$ with a color not in $C_{H}(z) \cup \{\vp(y_1), \vp(y_2)\}$, guaranteeing two unique colors for $y$.

If $\vp(z) \notin \{\vp(x_1), \vp(x_2)\}$, then $x$ has three unique colors, so coloring $y$ with a color not in $C_{H}(y) \cup \{\vp(z)\}$ guarantees at least two unique colors for $x$. 
If $\vp(z) \in \{\vp(x_1), \vp(x_2)\}$, then color $y$ with a color not in $C_{H}(y) \cup \{\vp(x_1), \vp(x_2)\}$, guaranteeing an additional unique color for $x$.
Note that $N_G(x)\setminus \{y\}$ already has a unique color since $H$ is an $S$-reduced graph. 
In all cases, $\vp$ is a $2$-PCF $9$-coloring of $G$, which is a contradiction.
\end{proof}

\begin{figure}[h]
\centering
\begin{tikzpicture}[scale=0.9]
\filldraw[fill=black, draw=black] (-1,0) circle (0.05) node [above] {$x$};
\filldraw[fill=black, draw=black] (0,0) circle (0.05) node [above] {$y$};
\filldraw[fill=black, draw=black] (1,0) circle (0.05) node [above] {$z$};
\filldraw[fill=black, draw=black] (0,-1) circle (0.05) node [above left] {$y_1$};
\draw (-1,0) -- (1,0);
\draw (-1.5,0) -- (-1,0);
\draw (-1,-0.5) -- (-1,0);
\draw (1.5,0) -- (1,0);
\draw (1,-0.5) -- (1,0);
\draw (0,0) -- (0,-1);
\draw (0,-1) -- (-0.3,-1.3);
\draw (0,-1) -- (0,-1.4);
\draw (0,-1) -- (0.3,-1.3);
\draw (0,-2) node {(a)};
\end{tikzpicture}
\qquad 
\begin{tikzpicture}[scale=0.9]
\draw (-1.5,1) -- (2.5,1);
\draw (-1,0.5) -- (-1,1);
\draw (1,0) -- (1,1);
\draw (0,1) -- (0,0);
\draw (0,0) -- (-0.3,-0.3);
\draw (0,0) -- (0.3,-0.3);
\draw (1,0) -- (0.7,-0.3);
\draw (1,0) -- (1.3,-0.3);
\draw (2,1) -- (2,0.5);
\filldraw[fill=black, draw=black] (-1,1) circle (0.05) node [above] {$x$};
\filldraw[fill=black, draw=black] (0,1) circle (0.05) node [above] {$y$};
\filldraw[fill=black, draw=black] (1,1) circle (0.05) node [above] {$z$};
\filldraw[fill=black, draw=black] (2,1) circle (0.05) node [above] {$w$};
\filldraw[fill=black, draw=black] (0,0) circle (0.05) node [above left] {$y_1$};
\filldraw[fill=black, draw=black] (1,0) circle (0.05) node [above right] {$z_1$};
\draw (0.5,-1) node{(b)};
\end{tikzpicture}
\qquad 
\begin{tikzpicture} [scale=0.9]
\draw (0,0) -- (1,0) -- (1,1) -- (0,1) -- cycle;
\draw (0,1) -- (0,2);
\draw (1,1) -- (1,2);
\draw (1,0) -- (2,0);
\draw (1,0) -- (-1,0);
\draw (1,0) -- (1,-1);
\draw[dotted] (0,0) -- (0,-1);
\filldraw[fill=black, draw=black] (0,0) circle (0.05) node [above left] {$w$};
\filldraw[fill=black, draw=black] (1,0) circle (0.05) node [above right] {$z$};
\filldraw[fill=black, draw=black] (1,1) circle (0.05) node [right] {$y$};
\filldraw[fill=black, draw=black] (0,1) circle (0.05) node [left] {$x$};
\filldraw[fill=white, draw=black] (0,2) circle (0.05) node [left] {$x_1$};
\filldraw[fill=white, draw=black] (1,2) circle (0.05) node [right] {$y_1$};
\filldraw[fill=white, draw=black] (2,0) circle (0.05) node [right] {$z_1$};
\filldraw[fill=white, draw=black] (1,-1) circle (0.05) node [right] {$z_2$};
\filldraw[fill=white, draw=black] (-1,0) circle (0.05) node [left] {$w_1$};
\filldraw[fill=white, draw=black] (0,-1) circle (0.05) node [left] {$w_2$}; 
\draw (0.5,-1.5) node {(c)};
\end{tikzpicture}
\qquad
\begin{tikzpicture}[scale=0.9] 
\draw (-1.5,1) -- (1.5,1);
\draw (-1,1.5) -- (-1, 0);
\draw (-1.5,0) -- (1.5,0);
\draw (1,1.5) -- (1,0);
\draw (0,1) -- (0,-1);
\draw[dotted] (-1,0) -- (-1,-0.5);
\draw[dotted] (1,0) -- (1,-0.5);
\filldraw[fill=black, draw=black] (0,0) circle (0.05) node [below left] {$z$};
\filldraw[fill=black, draw=black] (1,0) circle (0.05) node [below right] {$w$};
\filldraw[fill=black, draw=black] (1,1) circle (0.05) node [above right] {$u$};
\filldraw[fill=black, draw=black] (0,1) circle (0.05) node [above] {$v$};
\filldraw[fill=black, draw=black] (-1,1) circle (0.05) node [above left] {$x$};
\filldraw[fill=black, draw=black] (-1,0) circle (0.05) node [below left] {$y$};
\filldraw[fill=white, draw=black] (0,-1) circle (0.05) node [right] {$z_1$};
\draw (0, -1.5) node {(d)};
\end{tikzpicture}
\caption{Figures for Claims~\ref{clm:333-4},~\ref{clm:3333},~\ref{clm:3344}, and~\ref{clm:4faces}}
\label{fig:3vx}
\end{figure}
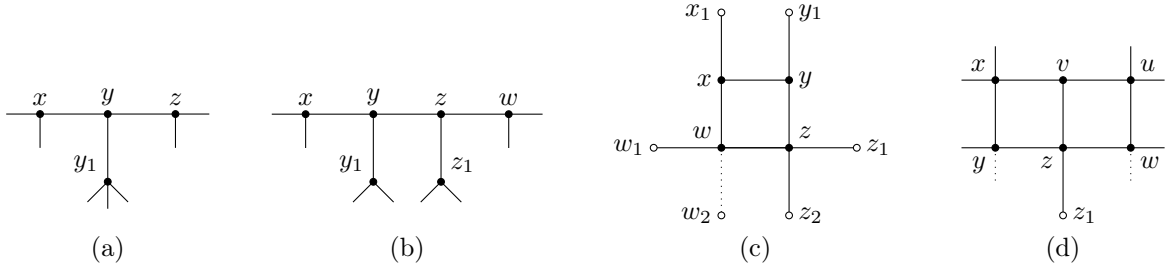

\begin{claim}\label{clm:333-4}
$G$ does not have a path on three $3$-vertices where the middle vertex is adjacent to a $4$-vertex.
\end{claim}

\begin{proof}
Suppose $G$ has a path $xyz$ on three $3$-vertices where   the neighbor $y_1$ of $y$ other than $x$ and $z$ is a $4$-vertex. 
See Figure~\ref{fig:3vx}(a).
For $S = \{x,y,z\}$, let $H$ be the $S$-reduced graph.
By the minimality of $G$, $H$ has a $2$-PCF $9$-coloring $\vp$.
Color $x$ with a color not in $C_H(x) \cup \{\vp(y_1)\}$, and color $z$ with a color not in $C_H(z) \cup \{\vp(x), \vp(y_1)\}$ to guarantee two (actually three) unique colors for $y$. 
Since $y_1$ is a 3-vertex in $H$, $\vp(N_G(y_1))$ consists of three distinct colors and at least two unique colors for $y_1$ are guaranteed regardless of the color assigned to $y$. 
Color $y$ with a color not in $\vp((N_G(x) \cup N_G(z)) \setminus S) \cup \{\vp(x), \vp(y_1), \vp(z)\}$ to obtain a $2$-PCF $9$-coloring $\vp$ of $G$, which is a contradiction.
\end{proof}

\begin{claim}\label{clm:3333}
$G$ does not have a path on four 3-vertices.
\end{claim}

\begin{proof}
Suppose $G$ has a path $xyzw$ on four $3$-vertices, and let $y_1$ (resp. $z_1$) be the neighbor of $y$ (resp. $z$) that is not on the path.
See Figure~\ref{fig:3vx}(b).
For $S = \{x,y,z,w\}$, let $H$ be the $S$-reduced graph.
By the minimality of $G$, $H$ has a $2$-PCF $9$-coloring $\vp$.
Note that $|C_H(x) \cup \{\vp(y_1)\}| \le 7$.
If $|C_H(x) \cup \{\vp(y_1)\}| = 7$, then color $y$ with a color in $\left(C_H(x) \cup \{\vp(y_1)\}\right) \setminus \left(C_H(y) \cup \vp((N_G(x) \setminus S)\cup \{z_1\})\right)$, 
and if $|C_H(x) \cup \{\vp(y_1)\}| \le 6$, then color $y$ with a color not in $C_H(y) \cup \vp((N_G(x)  \setminus S)\cup\{z_1\})$.
This guarantees two (actually three) unique colors for $x$, and in both cases, $|C_H(x) \cup \{\vp(y), \vp(y_1)\}|\le 7$.
Color $w$ with a color not in $C_H(w) \cup \{\vp(y), \vp(z_1)\}$ to guarantee two (actually three) unique colors for $z$, and color $z$ with a color not in $C_H(z) \cup \vp(N_G(w) \setminus S) \cup \{\vp(y), \vp(y_1), \vp(w)\}$ to guarantee two (actually three) unique colors for $w$. 
Finally, color $x$ with a color not in $C_H(x) \cup \{\vp(y), \vp(y_1), \vp(z)\}$ to guarantee two (actually three) unique colors for $y$, and now $\vp$ is a $2$-PCF $9$-coloring of $G$, which is a contradiction.
\end{proof}

\begin{claim}\label{clm:3344}
$G$ does not have a 4-cycle $xyzw$ where $x$ and $y$ are $3$-vertices. 
\end{claim}

\begin{proof}
Suppose $G$ has a $4$-cycle $F: xyzw$ where $x$ and $y$ are $3$-vertices. 
Let $x_1,y_1,z_1,w_1$ be a neighbor of $x,y,z,w$, respectively, that is not on $F$. 
By Claim~\ref{clm:3333}, we may assume $z$ is a $4$-vertex, so let $z_2$ be the neighbor of $z$ that is neither on $F$ nor $z_1$, and if $w$ is a $4$-vertex, then let $w_2$ be the neighbor of $w$ that is neither on $F$ nor $w_1$. 
See Figure~\ref{fig:3vx}(c). 
For $S = \{x,y,z,w\}$, let $H$ be the $S$-reduced graph.
By the minimality of $G$, $H$ has a $2$-PCF $9$-coloring $\vp$.
Color $z$ with a color not in $C_H(z) \cup \{\vp(y_1), \vp(w_1)\}$, and color $w$ with a color not in $C_H(w) \cup \{\vp(z), \vp(x_1)\}$.
Color $y$ with a color not in $C_H(y) \cup \{\vp(x_1), \vp(w), \vp(z), \vp(z_1), \vp(z_2)\}$ to guarantee two unique colors for each of $x$ and $z$.
Finally, color $x$ with a color not in $C_H(x) \cup \{\vp(y), \vp(y_1), \vp(z), \vp(w), \vp(w_1)\}$ to guarantee two unique colors for each of $y$ and $w$. 
Now, $\vp$ is a $2$-PCF $9$-coloring of $G$, which is a contradiction.
\end{proof}

\begin{claim}\label{clm:4faces}
$G$ does not have a $3$-vertex incident with two $4$-faces.
\end{claim}

\begin{proof}
Suppose $G$ has a $3$-vertex $v$ incident with two $4$-cycles $xyzv$ and $uwzv$.
By Claim~\ref{clm:3344}, $x,z,u$ are $4$-vertices. 
Let $z_1$ be the neighbor of $z$ that is not $y, v, w$. 
See Figure~\ref{fig:3vx}(d).
Let $H$ be the graph obtained from $G$ by removing $v$ and adding the edge $xu$, if it did not exist already.  
Note that $H$ is still planar and the maximum degree did not increase. 
By the minimality of $G$, $H$ has a $2$-PCF $9$-coloring $\vp$, so each of $\vp(N_G(x))$ and $\vp(N_G(u))$ must consist of at least two distinct colors. 
Let $\alpha$, $\beta$ be two distinct colors in $\vp(N_G(x))$, and let $\gamma, \delta$ be two distinct colors in $\vp(N_G(u))$.
Note that there are three unique colors for $z$, so regardless of the color assigned to $v$, at least two unique colors are guaranteed for $z$.

If $\vp(N_G(v))$ consists of three distinct colors, then  color $v$ with a color not in $\{\vp(x), \alpha, \beta, \vp(z), \vp(u), \gamma, \delta\}$ to guarantee two unique colors for each of $x$ and $u$. 
Now $\vp$ is a $2$-PCF $9$-coloring of $G$, which is a contradiction.
Thus, $\vp(N_G(v))$ consists of two distinct colors.

Without loss of generality, assume $\vp(x)=\vp(z)$.
There must be two unique colors for $y$, so $y$ must be a $4$-vertex and the two neighbors of $y$ other than $x$ and $z$ received different colors that is also different from $\vp(x)$. 
Thus, regardless of the color (re)assigned to $z$, two unique colors are guaranteed for $y$.
Note that since $\vp(N_G(w))$ contains at least three colors, there is a color $a \in \vp(N_G(w)) \setminus \{\vp(u),\vp(z)\}$.
Let $b$ and $c$ be two distinct colors in $\vp(N_G(z_1) \setminus \{z\})$.
Recolor $z$ with a color not in $\{\vp(x), \vp(u), \vp(y), \vp(z_1), b,c, \vp(w), a\}$ to guarantee two  unique colors for each of $z_1$ and $w$, and color $v$ with a color not in $\{\vp(x), \alpha, \beta, \vp(z), \vp(u), \gamma, \delta\}$ to guarantee two unique colors for each of $x$ and $u$.
Now, $\vp$ is a $2$-PCF $9$-coloring of $G$, which is a contradiction.
\end{proof}

\begin{figure}
    \centering
\begin{tikzpicture} [scale=0.7]
\filldraw[fill=black, draw=black] (-54:1) circle (0.05) node [below left] {$x$};
\filldraw[fill=black, draw=black] (18:1) circle (0.05) node [below right] {$y$};
\filldraw[fill=black, draw=black] (90:1) circle (0.05) node [above right] {$z$};
\filldraw[fill=black, draw=black] (162:1) circle (0.05) node [below left] {$u$};
\filldraw[fill=black, draw=black] (234:1) circle (0.05) node [below right] {$v$};
\filldraw[fill=black, draw=black] (18:2) circle (0.05) node [above left] {$y_1$};
\filldraw[fill=black, draw=black] (162:2) circle (0.05) node [above right] {$u_1$};
\draw (-54:1) -- (18:1) -- (90:1) -- (162:1) -- (234:1) -- cycle;
\draw \foreach \x in {18, 162} {(\x:1) -- (\x:2.6)};
\draw (-54:1) -- (-44:1.5);
\draw (-54:1) -- (-64:1.5);
\draw (18:2) -- (28:2.5);
\draw (18:2) -- (8:2.5);
\draw (162:2) -- (172:2.5);
\draw (162:2) -- (152:2.5);
\draw (234:1) -- (224:1.5);
\draw (234:1) -- (244:1.5);
\draw (90:1) -- (90:1.6);
\draw (0,-2) node {(a)};
\end{tikzpicture}    
\qquad
\begin{tikzpicture}[scale=0.7] 
\filldraw[fill=black, draw=black] (-54:1) circle (0.05) node [below left] {$x$};
\filldraw[fill=black, draw=black] (18:1) circle (0.05) node [below right] {$y$};
\filldraw[fill=black, draw=black] (90:1) circle (0.05) node [above right] {$z$};
\filldraw[fill=black, draw=black] (162:1) circle (0.05) node [below left] {$u$};
\filldraw[fill=black, draw=black] (234:1) circle (0.05) node [below right] {$v$};
\draw (-54:1) -- (18:1) -- (90:1) -- (162:1) -- (234:1) -- cycle;
\draw \foreach \x in {-54,90,234} {(\x:1) -- (\x:2)};
\draw (18:1) -- (28:1.7);
\draw (18:1) -- (8:1.7);
\draw (162:1) -- (152:1.7);
\draw (162:1) -- (172:1.7);
\draw (234:2) -- (224:2.3);
\draw (234:2) -- (244:2.3);
\draw (-54:2) -- (-44:2.3);
\draw (-54:2) -- (-64:2.3);
\draw (234:2) -- (234:2.4);
\draw (-54:2) -- (-54:2.4);
\filldraw[fill=black, draw=black] (234:2) circle (0.05) node [above left] {$v_1$};
\filldraw[fill=black, draw=black] (-54:2) circle (0.05) node [above right] {$x_1$};
\filldraw[fill=white, draw=black] (90:2) circle (0.05) node [below right] {$z_1$};
\draw (0,-2) node {(b)};
\end{tikzpicture}  
\qquad
\begin{tikzpicture}[scale=0.7] 
\filldraw[fill=black, draw=black] (-54:1) circle (0.05) node [below left] {$x$};
\filldraw[fill=black, draw=black] (18:1) circle (0.05) node [below right] {$y$};
\filldraw[fill=black, draw=black] (90:1) circle (0.05) node [above right] {$z$};
\filldraw[fill=black, draw=black] (162:1) circle (0.05) node [below left] {$u$};
\filldraw[fill=black, draw=black] (234:1) circle (0.05) node [below right] {$v$};
\draw (-54:1) -- (18:1) -- (90:1) -- (162:1) -- (234:1) -- cycle;
\draw \foreach \x in {-54,234} {(\x:1) -- (\x:2.4)};
\draw (18:1) -- (28:1.7);
\draw (18:1) -- (8:1.7);
\draw (162:1) -- (147:2);
\draw (162:1) -- (172:1.7);
\draw (90:1) -- (110:2);
\draw (110:2) -- (147:2);
\draw (110:2) -- (115:2.5);
\draw (110:2) -- (100:2.3);
\draw (234:2) -- (224:2.3);
\draw (234:2) -- (244:2.3);
\draw (-54:2) -- (-44:2.3);
\draw (-54:2) -- (-64:2.3);
\filldraw[fill=black, draw=black] (234:2) circle (0.05) node [above left] {$v_1$};
\filldraw[fill=black, draw=black] (-54:2) circle (0.05) node [above right] {$x_1$};
\filldraw[fill=black, draw=black] (110:2) circle (0.05) node [right] {$z_1$};
\filldraw[fill=white, draw=black] (147:2) circle (0.05) node [left] {$u_1$};
\draw (0,-2) node {(c)};
\end{tikzpicture}    
\qquad
\begin{tikzpicture} [scale=0.7]
\filldraw[fill=black, draw=black] (-54:1) circle (0.05) node [below left] {$x$};
\filldraw[fill=black, draw=black] (18:1) circle (0.05) node [below right] {$y$};
\filldraw[fill=black, draw=black] (90:1) circle (0.05) node [above right] {$z$};
\filldraw[fill=black, draw=black] (162:1) circle (0.05) node [above] {$u$};
\filldraw[fill=black, draw=black] (234:1) circle (0.05) node [below right] {$v$};
\draw (-54:1) -- (18:1) -- (90:1) -- (162:1) -- (234:1) -- cycle;
\draw \foreach \x in {-54,90} {(\x:1) -- (\x:2.4)};
\draw (234:1) -- (224:2);
\draw (18:1) -- (28:1.7);
\draw (18:1) -- (8:1.7);
\draw (162:1) -- (152:1.7);
\draw (162:1) -- (190:2);
\draw (-54:2) -- (-44:2.3);
\draw (-54:2) -- (-64:2.3);
\draw (190:2) -- (224:2);
\draw (224:2) -- (234:2.3);
\draw (224:2) -- (220:2.5);
\draw (90:2) -- (80:2.3);
\draw (90:2) -- (100:2.3);
\filldraw[fill=white, draw=black] (190:2) circle (0.05) node [left] {$u_1$};
\filldraw[fill=black, draw=black] (224:2) circle (0.05) node [below right] {$v_1$};
\filldraw[fill=black, draw=black] (-54:2) circle (0.05) node [above right] {$x_1$};
\filldraw[fill=black, draw=black] (90:2) circle (0.05) node [below right] {$z_1$};
\draw (0,-2) node {(d)};
\end{tikzpicture}    
    \caption{Figures for Claims~\ref{clm:33344},~\ref{clm:34334-3}, and~\ref{clm:bad3vx}}
    \label{fig:pentagon}
\end{figure}
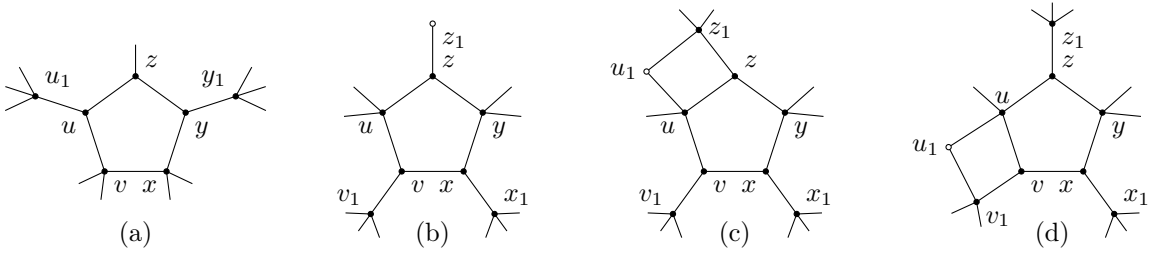

\begin{claim}\label{clm:33344}
$G$ does not have a $5$-cycle with three consecutive $3$-vertices.
\end{claim}
\begin{proof}
Suppose $G$ has a $5$-cycle $F:xyzuv$ with three consecutive $3$-vertices $y$, $z$, and $u$.
Let $y_1$ and $u_1$ be the neighbor of $y$ and $u$, respectively, that is not on $F$.
By Claim~\ref{clm:3333}, $x,v, y_1, u_1$ are $4$-vertices. 
See Figure~\ref{fig:pentagon}(a).
For $S = \{x,y,z,u,v\}$, let $H$ be the $S$-reduced graph.
By the minimality of $G$, $H$ has a $2$-PCF $9$-coloring $\vp$.
Color $v$ with a color not in $C_H(v) \cup \{\vp(u_1)\}$, and color $x$ with a color not in $C_H(x) \cup \{\vp(v), \vp(y_1)\}$.
Color $z$ with a color not in $C_H(z) \cup \{\vp(u_1), \vp(v), \vp(y_1), \vp(x)\}$ to guarantee two (actually three) unique colors for each of $y$ and $u$.
Color $y$ with a color not in $C_H(y) \cup \vp((N_G(x) \cup N_G(z)) \setminus S) \cup \{\vp(z), \vp(x)\}$ to guarantee two unique colors for $x$.
Note that $u_1$ already has three unique colors, so regardless of the color assigned to $u$, at least two unique colors are guaranteed for $u_1$.
Color $u$ with a color not in $\vp((N_G(v) \cup N_G(z)) \setminus S) \cup \{\vp(y), \vp(z),\vp(v), \vp(u_1)\}$ to guarantee two unique colors for each of $z$ and $v$.
Now, $\vp$ is a $2$-PCF $9$-coloring of $G$, which is a contradiction.
\end{proof}

\begin{claim}\label{clm:34334-3}
If $G$ has a $5$-cycle $F$ incident with three $3$-vertices, then every $3$-vertex on $F$ has a $4$-neighbor that is not on $F$.  
\end{claim}
\begin{proof}
Let $F:xyzuv$ be a $5$-cycle of $G$ incident with three $3$-vertices.
By Claim~\ref{clm:33344}, we may assume $x, z, v$ are $3$-vertices and $y,u$ are $4$-vertices.
Let $x_1$, $z_1$, and $v_1$ be the neighbor of $x$, $z$, and $v$, respectively, that is not on $F$.
See Figure~\ref{fig:pentagon}(b).
By Claim~\ref{clm:333-4}, $x_1$ and $v_1$ are $4$-vertices.

Suppose $z_1$ is a $3$-vertex.
For $S = \{x,y,z,u,v,z_1\}$, let $H$ be the $S$-reduced graph.
By the minimality of $G$, $H$ has a $2$-PCF $9$-coloring $\vp$.
Color $y$ with a color not in $C_H(y) \cup \{\vp(x_1)\}$, and color $u$ with a color not in $C_H(u) \cup \{\vp(y), \vp(v_1)\}$.
Color $z_1$ with a color not in $C_H(z_1) \cup \{\vp(y), \vp(u)\}$ to guarantee two (actually three) unique colors for $z$.
Color $z$ with a color not in $\vp(N_G(z_1) \setminus S) \cup \{\vp(y), \vp(u), \vp(z_1)\}$ to guarantee two (actually three) unique colors for $z_1$.
Note that there are three unique colors for each of $x_1$ and $v_1$, so regardless of the color assigned to $x$ and $v$, at least two unique colors are guaranteed for $x_1$ and $v_1$.
Color $x$ with a color not in $\vp(N_G(y) \setminus S) \cup \{\vp(y), \vp(u), \vp(v_1), \vp(x_1)\}$ to guarantee two unique colors for each of $y$ and $v$.
Color $v$ with a color not in $\vp(N_G(u) \setminus S) \cup \{\vp(u), \vp(v_1), \vp(x), \vp(x_1), \vp(y)\}$ to guarantee two unique colors for each of $x$ and $u$.
Now, $\vp$ is a $2$-PCF $9$-coloring of $G$, which is a contradiction.
\end{proof}

A $3$-vertex on a $4$-cycle is {\it bad}, and a $3$-vertex on no $4$-cycles is {\it good}.

\begin{claim}\label{clm:bad3vx}
If $G$ has a $5$-cycle $F$ incident with three $3$-vertices, then every $3$-vertex on $F$ is a good $3$-vertex.
\end{claim}
\begin{proof}
Let $F:xyzuv$ be a $5$-cycle with three $3$-vertices.
By Claim~\ref{clm:33344}, we may assume $x$, $z$, $v$ are $3$-vertices and $y$, $u$ are $4$-vertices.
Let $x_1$, $z_1$, and $v_1$ be the neighbor of $x$, $z$, and $v$, respectively, that is not on $F$.
By Claim~\ref{clm:34334-3}, $x_1$, $z_1$, and $v_1$ are $4$-vertices.

Suppose $z$ is a bad $3$-vertex. 
Without loss of generality,  assume  $uzz_1u_1$ is a $4$-cycle where $u_1$ is a neighbor of $u$ not on $F$. 
See Figure~\ref{fig:pentagon}(c).
For $S=\{x,y,z,u,v,z_1,u_1\}$, let $H$ be the $S$-reduced graph.
By the minimality of $G$, $H$ has a $2$-PCF $9$-coloring $\vp$.
Color $z_1$ with a color not in $C_H(z_1) \cup \vp(N_G(u_1) \setminus S)$, color $y$ with a color not in $C_H(y) \cup \{\vp(x_1),\vp(z_1)\}$, and color $u$ with a color not in $C_H(u) \cup \vp(N_G(u_1) \setminus S) \cup \{\vp(v_1), \vp(y),\vp(z_1)\}$ to guarantee two unique colors for each of $z$ and $u_1$.
Color $u_1$ with a color not in $C_H(u_1) \cup \{\vp(z_1), \vp(u)\}$, and color $z$ with a color not in $\vp(N_G(z_1)\setminus S) \cup \{\vp(z_1), \vp(u), \vp(y), \vp(u_1)\}$.
At this point,  two unique colors for $z_1$ are guaranteed if $z_1 \notin \{x_1,v_1\}$; if $z_1 \in \{x_1, v_1\}$, then two unique colors for $z_1$ will be guaranteed when coloring $x$ and $v$.
Color $x$ with a color not in $\vp(N_G(y) \setminus S) \cup \{\vp(y),\vp(x_1),\vp(u_1),\vp(z),\vp(v_1), \vp(u)\}$ to guarantee two unique colors for each of $y$ and $v$, and also $z_1$ if $x_1=z_1$.
Note that if $x_1\neq z_1$, then there are three unique colors for $x_1$, so regardless of the color assigned to $x$, at least two unique colors are guaranteed for $x_1$.
Finally, color $v$ with a color not in $\{\vp(x), \vp(x_1),\vp(y), \vp(u), \vp(u_1), \vp(z), \vp(v_1)\}$ to guarantee two unique colors for each of $x$ and $u$, and also $z_1$ if $v_1=z_1$.
Note that if $v_1\neq z_1$, then there are three unique colors for $v_1$, so regardless of the color assigned to $v$, at least two unique colors are guaranteed for $v_1$.
Now, $\vp$ is a $2$-PCF $9$-coloring of $G$, which is a contradiction.


Suppose $v$ or $x$ is a bad $3$-vertex.
Without loss of generality, assume  $uvv_1u_1$ is a $4$-cycle where $u_1$ is a neighbor of $u$ not on $F$. 
See Figure~\ref{fig:pentagon}(d).
For $S'=\{x,y,z,u,v,u_1,v_1\}$, let $H'$ be the $S'$-reduced graph.
By the minimality of $G$, $H'$ has a $2$-PCF $9$-coloring $\vp'$.
Color $y$ with a color not in $C_{H'}(y) \cup \{\vp'(z_1), \vp'(x_1)\}$, 
color $v_1$ with a color not in $C_{H'}(v_1) \cup \vp'(N_G(u_1) \setminus S')$,
and color $u$ with a color not in $C_{H'}(u) \cup \vp'(N_G(u_1) \setminus S') \cup \{\vp'(y), \vp'(z_1), \vp'(v_1)\}$ to guarantee two unique colors for each of $z$ and $u_1$.
Color $u_1$ with a color not in $C_{H'}(u_1) \cup \{\vp'(u), \vp'(v_1)\}$,
and color $v$ with a color not in $\vp'(N_G(v_1) \setminus S') \cup \{\vp'(v_1), \vp'(u_1), \vp'(u), \vp'(x_1), \vp'(y)\}$ to guarantee two unique colors for $x$.
At this point two unique colors for $v_1$ are guaranteed if $v_1 \neq z_1$; if $v_1 = z_1$, then two unique colors for $v_1$ will be guaranteed when coloring $z$.
Color $x$ with a color not in $C_{H'}(x) \cup \{\vp'(v), \vp'(v_1), \vp'(u), \vp'(y)\}$ to guarantee two (actually three) unique colors for $v$.
Color $z$ with a color not in $\vp'(N_G(y) \setminus S') \cup \{\vp'(y),\vp'(z_1), \vp'(u), \vp'(u_1), \vp'(v)\}$ to guarantee two unique colors for each of $u$ and $y$, and also $v_1$ if $z_1=v_1$.
Note that if $z_1\neq v_1$, then there are three unique colors for $z_1$, so regardless of the color assigned to $z$, at least two unique colors are guaranteed for $z_1$.
Now, $\vp'$ is a $2$-PCF $9$-coloring of $G$, which is a contradiction.
\end{proof}

Using the above claims, we now explicitly state and prove the essential reducible configurations. 

\begin{lemma}\label{lem:reducible}
In $G$,
\begin{itemize}
\item[(1)] every vertex has degree at least $3$,
\item[(2)] every cycle has length at least $4$,
\item[(3)] every $3$-vertex is incident with at most one $4$-face,
\item[(4)] if a $5$-face is incident with exactly three $3$-vertices, then they are all good $3$-vertices.
\item[(5)] every $5^+$-face $f$ is incident with at most $\left \lfloor \frac{3d(f)}{4} \right \rfloor$ $3$-vertices.
\end{itemize}
\end{lemma}
\begin{proof}

By Claim~\ref{clm:3vx}, every vertex has degree at least $3$ so (1) holds.
By Claim~\ref{clm:triangle}, every cycle has length at least $4$ so (2) holds.
By Claim~\ref{clm:4faces}, every $3$-vertex is incident with at most one $4$-face so  (3) holds.
By Claim~\ref{clm:bad3vx}, if a $5$-face is incident with exactly three $3$-vertices, then they are all good $3$-vertices, hence (4) holds.
By Claim~\ref{clm:3333}, every $5^+$-face $f$ does not have four  consecutive $3$-vertices, so $f$ is incident with at most $\left \lfloor \frac{3d(f)}{4} \right \rfloor$ $3$-vertices, hence (5) holds.
\end{proof}

We use the well-known discharging method to finish off the proof. 
See~\cite{2017CrWe} for a nice expository survey of the method. 
Let $F(G)$ denote the set of faces of $G$, and for a face $f$, let $d(f)$ denote the length of a boundary walk of $f$. 
For each $z\in V(G)\cup F(G)$, let the {\it initial charge} $\mu(z)$ of $z$ be $d(z)-4$.
By Euler's formula the sum of all initial charge is negative: $\sum_{v\in V(G)}(d(v)-4)+\sum_{f\in F(G)}(d(f)-4)=2|E(G)|-4|V(G)|+2|E(G)|-4|F(G)|=-8$.


Here are the discharging rules: 

\begin{enumerate}
    \item[\textbf{[R1]}] Every $5$-face sends charge $1/3$ to each incident good $3$-vertex.
    \item[\textbf{[R2]}] Every $5$-face sends charge $1/2$ to each incident bad $3$-vertex.
    \item[\textbf{[R3]}] Every $6^+$-face sends charge $1/2$ to each incident $3$-vertex.
\end{enumerate}

We recount the charge after applying the discharging rule.
We will obtain that the final charge is non-negative for each vertex and face, to conclude that the sum of the final charge is non-negative. 
This is a contradiction since the initial charge sum is negative and the discharging rule preserved the total charge sum.
We conclude that a counterexample could not have existed in the first place.

Only $3$-vertices have negative initial charge since $G$ has no $2^-$-vertices by Lemma~\ref{lem:reducible}~(1).
Note that $G$ has no $3$-faces by Lemma~\ref{lem:reducible}~(2).

Each good $3$-vertex $v$ is incident with three $5^+$-faces, each of which sends charge at least $\frac{1}{3}$ to $v$ by \textbf{[R1]} and \textbf{[R3]}, so the final charge of $v$ is at least $-1+\frac{1}{3}\cdot3 = 0$.
Each bad $3$-vertex $v$ is incident with at least two $5^+$-faces by Lemma~\ref{lem:reducible}~(3), so $v$ receives charge $\frac{1}{2}$ at least twice by \textbf{[R2]} and \textbf{[R3]}, so the final charge of $v$ is at least $-1+\frac{1}{2}\cdot2=0$. 
Each $4$-vertex and $4$-face is not involved in the discharging process, so the final charge is the initial charge, which is $0$.
If $f$ is a $5$-face incident with  a bad $3$-vertex, then $f$ is incident with at most one other  $3$-vertex by Lemma~\ref{lem:reducible}~(4) and (5), so  the final charge of $f$ is at least $1-\frac{1}{2}\cdot2 =0$ by \textbf{[R1]} and \textbf{[R2]}.
If $f$ is a $5$-face not incident with a bad $3$-vertex, then $f$ is incident with at most three good $3$-vertices by Lemma~\ref{lem:reducible}~(5), so  the final charge of $f$ is at least $1-\frac{1}{3}\cdot 3=0$ by \textbf{[R1]}.
Each $6^+$-face $f$ has at most  $\left\lfloor \frac{3d(f)}{4}\right\rfloor$ incident $3$-vertices by Lemma~\ref{lem:reducible}~(5).
Thus, the final charge of $f$ is at least $d(f)-4 - \left\lfloor \frac{3d(f)}{4}\right\rfloor \frac{1}{2}$ by \textbf{[R2]}, which is non-negative since $d(f) \ge 6$.

 \section{Further discussion}

As mentioned in the introduction, Wegner's Planar Graph Conjecture is true for planar graphs with maximum degree $3$.
Recall that for a graph $G$ (not necessarily planar) with maximum degree $3$, properly coloring $G^2$ is equivalent to a $2$-PCF coloring of $G$. 
One could also ask what the $1$-PCF chromatic number is for planar graphs with maximum degree $3$, yet this is already known to be at most $4$ by a result of Liu and Yu~\cite{2013LiYu}.
Their result actually applies to all graphs (not necessarily planar) of maximum degree $3$; see also the discussion in the last section of~\cite{arXiv_CaPeSk}.
Caro, Petru\v{s}evski, and \v{S}krekovski~\cite{arXiv_CaPeSk} conjectured that every graph $G$ that is not the $5$-cycle is 1-PCF $(\Delta(G)+1)$-colorable; this conjecture is known to be true for only $\Delta(G)\leq 3$. 

For planar graphs with maximum degree $4$, Wegner's Planar Graph Conjecture is unresolved, so we proved a result in the flavor of 2-PCF colorings. 
One could also ask what the maximum 1-PCF chromatic number is for a planar graph with maximum degree $4$.
By the conjecture mentioned in the previous paragraph, one guess is that the bound is at most $5$. 

We also remark that in~\cite{fabrici2022proper}, Fabrici et al. constructed  a planar graph that is not 1-PCF 5-colorable,  conjectured that all planar graphs are 1-PCF 6-colorable, and proved that all planar graphs are 1-PCF 8-colorable.






\section*{Acknowledgements}
Eun-Kyung Cho was supported by Basic Science Research Program through the National Research Foundation of Korea (NRF) funded by the Ministry of Education (NRF-2020R1I1A1A0105858711).
Ilkyoo Choi was supported by the Basic Science Research Program through the National Research Foundation of Korea (NRF) funded by the Ministry of Education (NRF-2018R1D1A1B07043049), and also by the Hankuk University of Foreign Studies Research Fund.
Bernard Lidick\'y was supported in part by NSF grants DMS-1855653 and DMS-2152490.

\bibliographystyle{abbrvurl}
\bibliography{refs}

\end{document}